\documentclass[12pt]{amsart}
\usepackage{amsfonts, amsbsy, amsmath, amssymb}

\usepackage[T1]{fontenc}
\usepackage[utf8]{inputenc}
\usepackage{lmodern}
\usepackage{amsmath, amsthm, amssymb,amscd, mathrsfs, amsfonts, mathtools,tikz-cd}
\usepackage{relsize}
\usepackage{euler}
\usepackage{mathtools}

\usepackage{pbox}
\usepackage{times}
\usepackage[all]{xy}
\usepackage{todonotes}
\usepackage{xcolor}
\usepackage[bookmarks=true,
bookmarksnumbered=true, breaklinks=true,
pdfstartview=FitH, hyperfigures=false,
plainpages=false, naturalnames=true,
colorlinks=true,pagebackref=true,
pdfpagelabels]{hyperref}

\usepackage{tikz}
\usepackage{pgfplots}

\pgfplotsset{compat=1.10}
\usepgfplotslibrary{fillbetween}
\usetikzlibrary{patterns}
\usetikzlibrary{matrix}

\usepackage[lite,alphabetic]{amsrefs}

\renewcommand{\PrintDOI}[1]{\href{http://dx.doi.org/\detokenize{#1}}{doi: \detokenize{#1}}%
	\IfEmptyBibField{pages}{, (to appear in print)}{}}

\def\commutatif{\ar@{}[rd]|{\circlearrowleft}}

\newtheorem{theorem}{Theorem}[section]
\newtheorem{prop}[theorem]{Proposition}
\newtheorem{lem}[theorem]{Lemma}
\newtheorem{cor}[theorem]{Corollary}

\theoremstyle{definition}
\newtheorem{definition}[theorem]{Definition}

\theoremstyle{remark}
\newtheorem{rmk}[theorem]{Remark}

\newtheorem{example}[theorem]{Example}

\allowdisplaybreaks

\newcommand{\R}{\mathbb{R}}

\usepackage{bbm}

\def \R{\mathbb{R}}

\hypersetup{
	colorlinks = true,
	urlcolor = blue,
	linkcolor = blue,
	citecolor = red,
	pdfpagemode = UseNone
}

\begin{document}

\title{Helly Numbers for Connected Reductive Groups and Splitting of Toric Principal Bundles}

\begin{abstract}
For an almost simple complex algebraic group $G$ of semisimple rank $m$, we prove that a finite family of parabolic subgroups contains a common maximal torus whenever every subfamily of at most $m+2$ members does.  We determine the resulting Helly number $(h(G))$ for almost simple groups of types $A$, $B$, $C$, $G_2$, $F_4$, $E_7$,
and $E_8$, and obtain bounds differing by one in types $D$ and $E_6$. For general reductive groups of positive semisimple rank, the Helly number is the maximum of those of the almost simple factors. We apply these results to  establish a fan-theoretic criterion characterizing smooth toric varieties on which every toric principal $G$-bundle splits equivariantly. In particular, for $d \ge 2$, this splitting property holds on $\mathbb{P}^d$ if and only if $d \ge h(G)$.
\end{abstract}

\author{Boris Tsvelikhovskiy} 
\address{Department of Mathematics, University of California, Riverside, CA, USA} 
\email{tsvelibor@gmail.com}

\maketitle


\section{Introduction}

Let $W$ be a finite-dimensional vector space over a field $\Bbbk$. A collection of flags $\mathcal{F}_1,  \dots, \mathcal{F}_n$ in $W$ is said to be \textit{adapted} to a basis $\mathcal{B}$ of $W$ if each subspace appearing in any of the flags $\mathcal{F}_i$ is spanned by a subset of $\mathcal{B}$. A remarkable theorem of Klyachko \cite[Theorem~6.1.2]{Klyachko} establishes a local-to-global criterion for the simultaneous adaptation of flags.

\begin{theorem}
\label{thm:klyachko-helly-flags}
Let $S = \{\mathcal{F}_1, \dots, \mathcal{F}_n\}$ be a collection of flags in an $m$-dimensional vector space $W$. For $S$ to be simultaneously adapted to a single basis $\mathcal{B}$ of $W$, it is necessary and sufficient that every subcollection of at most $m+1$ flags from $S$ be adapted to some basis of $W$.
\end{theorem}

Following Klyachko, we refer to Theorem~\ref{thm:klyachko-helly-flags} as a \textit{Helly-type theorem} for vector subspaces and flags. It serves as a direct linear algebraic analogue to the classical Helly theorem in convex geometry, which asserts that a finite family of convex sets in $\mathbb{R}^m$ has nonempty intersection if and only if every subcollection of at most $m+1$ sets has nonempty intersection.

\subsection*{Parabolic Subgroups and Toric Bundles}

To generalize this result to arbitrary algebraic groups, we reformulate the notion of flag adaptation in group-theoretic terms. Recall that a flag $\mathcal{F}_i$ in a vector space $W$ corresponds to the parabolic subgroup $P_i \subset \operatorname{GL}(W)$ that stabilizes $\mathcal{F}_i$. Moreover, the simultaneous adaptation of a collection of flags to a basis $\mathcal{B}$ is equivalent to their corresponding parabolic subgroups sharing a common maximal torus $T \subset \operatorname{GL}(W)$ (namely, the torus diagonalized by $\mathcal{B}$). 

This formulation allows us to extend Helly-type results to arbitrary connected reductive algebraic groups over $\mathbb{C}$.

\begin{theorem}[See Theorem~\ref{thm:parabolicsGen}]
\label{thm:parabolics-intro}
Let $G$ be a connected reductive algebraic group over $\mathbb{C}$ of semisimple rank $m$. A finite family of parabolic subgroups of $G$ contains a common maximal torus if and only if every subfamily of at most $m+2$ members does.
\end{theorem}

To establish Theorem~\ref{thm:parabolics-intro}, we begin by giving a geometric proof of the corresponding statement for $\mathrm{GL}(W)$, recovering the $(m+1)$-threshold by recasting the argument of \cite[Theorem~6.2.1]{Klyachko}.

More precisely, given a family of $n+1$ flags in $W$, we associate these flags with the $n+1$ rays in the toric fan of $\mathbb{P}^n$. When every proper subfamily admits a common adapted basis, the resulting filtrations satisfy Klyachko's compatibility condition and thus define a toric vector bundle of rank $\dim (W)$ on $\mathbb{P}^n$. An equivariant splitting of this bundle is equivalent to the simultaneous adaptation of all $n+1$ flags.  

This formulation extends the argument from general linear groups to connected reductive groups by replacing vector bundles with equivariant principal $G$-bundles on $\mathbb P^n$.

\subsection*{Helly Numbers}

Theorem~\ref{thm:parabolics-intro} provides an upper bound on the Helly number $h(G)$ (introduced in \cite{HK}), defined as the minimal positive integer such that a finite family of parabolic subgroups of $G$ shares a common maximal torus whenever every subfamily of size at most $h(G)$ does.

For an almost simple algebraic group $G$ of semisimple rank $m$, we prove that $h(G) = m + 2$ when $G$ is of Lie type $A_m$, $B_m$, $C_m$, $G_2$, $F_4$, $E_7$, or $E_8$. For the remaining types, we obtain
\[
  m+1 \le h(G) \le m+2 \quad (\text{type } D_m, \, m \ge 4), \qquad 7 \le h(G) \le 8 \quad (\text{type } E_6)
\]
(see Corollary~\ref{cor:helly-values-by-type}). These lower bounds are established via explicit geometric configurations in classical groups and embeddings of reductive subgroups in exceptional groups.

Furthermore, we show that if $G$ has positive semisimple rank and $G_1, \dots, G_t$ denote the almost simple factors of the simply connected cover of its derived subgroup $G^{\mathrm{der}}$, then
\[
  h(G) = \max_{1 \le i \le t} h(G_i)
\]
(see Proposition~\ref{prop:helly-central-products}). Thus, the general upper bound in Theorem~\ref{thm:parabolics-intro} can be sharpened to $2 + \max\limits_{1 \le i \le t} \operatorname{rank}(G_i)$, demonstrating that the optimal threshold is governed by the maximal rank among the irreducible simple factors rather than the total semisimple rank of $G$.

\subsection*{Equivariant Splitting and Helly-Type Thresholds}

While the equivalence between the category of toric vector bundles on $X_\Sigma$ and the category of finite-dimensional vector spaces equipped with collections of compatible decreasing filtrations for each ray is used to establish our main result (Theorem~\ref{thm:parabolics-intro}), this correspondence can also be leveraged to establish converse implications and yield complete criteria for the equivariant splitting of toric principal $G$-bundles on smooth toric varieties—that is, the existence of an equivariant reduction of the structure group to a maximal torus $T \subset G$.

Specifically, let $X_\Sigma$ be a smooth toric variety and set $s = |\Sigma(1)|$ to be the number of rays in $\Sigma$. If $s \ge 3$, every toric principal $G$-bundle on $X_\Sigma$ splits equivariantly if and only if every collection of $\min\{s, h(G)\}$ distinct rays in $\Sigma(1)$ is contained in a cone of $\Sigma$. If $s \le 2$, all such bundles split automatically. In particular, this implies that for $d \ge 2$, every toric principal $G$-bundle on $\mathbb{P}^d$ splits equivariantly if and only if $d \ge h(G)$, whereas on $\mathbb{P}^1$ all such bundles split (see Corollary~\ref{cor:exact-projective-splitting}).
We also obtain analogues of Klyachko's dimension criteria (Corollary~\ref{cor:klyachko-dimension-principal}). Let $X_\Sigma$ be a smooth $n$-dimensional toric variety with $|\Sigma(1)| \ge 3$. When $n = h(G) - 1$, every toric principal $G$-bundle on $X_\Sigma$ splits equivariantly if and only if $X_\Sigma$ is affine. Similarly, when $n = h(G)$, every such bundle splits equivariantly if and only if $X_\Sigma$ is affine or torically isomorphic to $\mathbb{P}^n$.

A deeper interplay between the filtration-theoretic and bundle-theoretic perspectives yields sharper bounds when restricting to specific parabolic types. If the Levi factors of the parabolics under consideration have semisimple rank at most $\ell$, testing subfamilies of at most $\ell + 3$ members suffices (see Proposition~\ref{prop:levi-rank-helly}). In particular, $3$ is the optimal Helly threshold for Borel subgroups whenever $G$ has positive semisimple rank. For partial flags in $\mathbb{C}^m$ with at least $s$ nonzero successive quotients (from $0$ to $\mathbb{C}^m$), we obtain the optimal threshold

$$m - s + 3, \qquad 2 \le s \le m.$$

This result bridges the two extreme cases, shifting from $m + 1$ for single subspaces ($s = 2$) down to $3$ for complete flags ($s = m$) (see Corollary \ref{cor:partial-flag-length-helly}).

\subsection*{Helly Numbers for Spherical Buildings}

Finally, we interpret the main results in the language of buildings. In the spherical building of $G$, parabolic subgroups correspond to simplices, and sharing a maximal torus is equivalent to lying in a common apartment. For a connected reductive complex algebraic group $G$ of positive semisimple rank $m$, its apartment Helly number (see Definition~\ref{defn:apartment-helly}) is therefore equal to $h(G)$ (see Corollary~\ref{cor:spherical-building-helly}).

\subsection*{Organization of the Paper}
The remainder of this paper is organized as follows. In Section~\ref{sec:review}, we review Klyachko's classification of toric vector bundles and present our bundle-theoretic proof of the flag adaptation theorem for $\operatorname{GL}(W)$. Section~\ref{sec:reductive_groups} extends the restriction argument to connected reductive groups, establishing the semisimple-rank bound for their Helly numbers. Section~\ref{subsec:helly-by-type} investigates the optimal Helly number, computes it for several almost simple types, and analyzes its behavior under central isogenies and direct products. Section~\ref{sec:impl_for_toric_G-bundles} derives splitting criteria for toric principal bundles along with sharper Helly bounds for prescribed parabolic and flag types.  Finally, Section~\ref{sec:buildings} is devoted to interpretation of these results within the framework of spherical buildings, determining the corresponding apartment Helly numbers.

\subsubsection*{Acknowledgement.}
I am grateful to Kiumars Kaveh for introducing me to the subject and to the paper \cite{Klyachko}, which served as both a source of knowledge and inspiration. I am also indebted to Kiumars Kaveh and Christopher Manon for their interest in this work and for many fruitful discussions.
 
\section{The \texorpdfstring{$\operatorname{GL}(V)$}{GL(V)} Case: Review and Geometric Proof}
\label{sec:review}

Throughout this section, let $\Bbbk$ be an algebraically closed field and $V$ an $m$-dimensional vector space over $\Bbbk$. A collection of subspaces $\{W_1, W_2, \dots, W_n\}$ of $V$ is said to be \emph{adapted} to a basis $\mathcal{B}$ of $V$ if each subspace $W_i$ is spanned by a subset of $\mathcal{B}$. A fundamental result of Klyachko \cite[Theorem~6.1.2]{Klyachko} establishes a local-to-global criterion for subspace adaptation:

\begin{theorem}
\label{thm:Klyachko}
Let $\mathcal{S} = \{W_1, \dots, W_n\}$ with $n\ge m+1$ be a family of subspaces in an $m$-dimensional vector space $V$. The collection $\mathcal{S}$ is adapted to a common basis of $V$ if and only if every subcollection of $m+1$ subspaces from $\mathcal{S}$ is adapted to some basis of $V$.
\end{theorem}

Following Klyachko, we refer to this theorem as a Helly-type theorem for subspaces of a vector space. It is similar, in spirit, to the classical Helly theorem which states that given a finite collection of convex subsets in $\R^m$, if any $m+1$ convex sets in this collection have nonempty intersection then all of the sets in the collection have nonempty intersection.

By interpreting flags through their stabilizers in the general linear group $\operatorname{GL}(V)$, Theorem~\ref{thm:Klyachko} admits an equivalent formulation in the language of algebraic groups. Specifically, a choice of basis $\mathcal{B}$ for $V$ uniquely determines a maximal torus $T \subset \operatorname{GL}(V)$ consisting of operators diagonal with respect to $\mathcal{B}$. A flag $\mathcal{F}$ in $V$ is adapted to $\mathcal{B}$ if and only if its stabilizer $P = \operatorname{Stab}_{\operatorname{GL}(V)}(\mathcal{F})$ contains $T$. Since the stabilizer of any flag in $\operatorname{GL}(V)$ is a parabolic subgroup, Klyachko's theorem can be reformulated in terms of parabolic subgroups.

To extend the statement from individual subspaces to full flags, we consider the collection formed by taking the union of the subspaces comprising these flags. As any subfamily containing at most $m+1$ such subspaces originates from at most $m+1$ distinct flags, Theorem~\ref{thm:Klyachko} applies directly to this union.

\begin{theorem}
\label{thm:parabolics}
Let \(V\) be an \(m\)-dimensional vector space over an algebraically closed field \(\Bbbk\), with \(m\geq 1\), and let \(\{P_\alpha\}_{\alpha\in A}\) be a finite family of parabolic subgroups of \(\operatorname{GL}(V)\). Then \(\bigcap\limits_{\alpha\in A}P_\alpha\) contains a maximal torus if and only if the intersection of every subfamily of at most \(m+1\) members contains a maximal torus.
\end{theorem}

We give a geometric proof of Theorem \ref{thm:parabolics} using Klyachko's classification of toric vector bundles and Horrocks' splitting criterion. While the original proof presented in \cite{Klyachko} is framed differently, the core argument proposed here closely follows Step 3 of the proof of Theorem $6.2.1$ therein. This approach admits a natural generalization to the setting of connected reductive complex algebraic groups, as explored in subsequent sections.

\subsection{Toric vector bundles and Klyachko filtrations}
\label{subsec:klyachko_setup}

We start with a brief review of the foundational setup (for a more detailed exposition, we refer the reader to \cite{Klyachko,Kaneyama} as well as Section 4 in \cite{KS}).

Let $T \cong \mathbb{G}_m^n$ denote an $n$-dimensional algebraic torus over an algebraically closed field $\Bbbk$. Let 
\begin{equation}
    \label{eq:cochar_lattice}
    N = \operatorname{Hom}(\mathbb{G}_m, T)
\end{equation}
be its lattice of one-parameter subgroups, which is dual to the character lattice
\[
M = \operatorname{Hom}(T, \mathbb{G}_m),
\]
under the canonical pairing
\[
\langle\cdot,\cdot\rangle \colon M \times N \longrightarrow \mathbb{Z}.
\]
Recall that a normal toric variety $X$ is an irreducible normal algebraic variety containing $T$ as a dense open orbit, such that the action of $T$ on itself extends to an action on $X$. Up to $T$-equivariant isomorphism, $X$ is uniquely determined by a finite rational polyhedral fan $\Sigma$ in $N_{\mathbb{R}} \coloneqq N \otimes_{\mathbb{Z}} \mathbb{R}$, motivating the notation $X_\Sigma$.

\begin{definition}
A \emph{toric vector bundle} on a toric variety $X_\Sigma$ with acting torus $T$ is a vector bundle $\pi \colon \mathcal{E} \to X_\Sigma$ equipped with a $T$-action on the total space $\mathcal{E}$ that is linear on fibers and $T$-equivariant with respect to $\pi$. That is, for every $t \in T$, the maps on fibers $\mathcal{E}_x \to \mathcal{E}_{t \cdot x}$ are linear, and the following diagram commutes:

\[
\begin{tikzcd}[row sep=large, column sep=large]
 \mathcal{E} \arrow[r, "t"] \arrow[d, "\pi"'] & \mathcal{E} \arrow[d, "\pi"] \\
 X_\Sigma \arrow[r, "t"'] & X_\Sigma
\end{tikzcd}
\]    
\end{definition}

We now recall Klyachko's classification of toric vector bundles \cite{Klyachko}. Fix a base point $x_0 \in T \subset X_\Sigma$ (representing the identity element of $T$) in the open torus orbit and let
\begin{equation}
\label{eq:toric_fiber}
    E \coloneqq \mathcal{E}_{x_0}
\end{equation}
be the fiber over $x_0$.
For each cone $\sigma \in \Sigma$, let $X_\sigma \subset X_\Sigma$ denote the corresponding $T$-invariant open affine chart. We use the action on sections given by
\[
  (t\cdot s)(x)=t\cdot s(t^{-1}x).
\]
Thus the character function $\chi^{u'}$ has weight $-u'$. The space of sections $\Gamma(X_\sigma, \mathcal{E})$ is a $T$-module, which decomposes into weight spaces:
\begin{equation}
\label{eq:weight_space_decomposition}
    \Gamma(X_\sigma, \mathcal{E}) = \bigoplus_{u \in M} \Gamma(X_\sigma, \mathcal{E})_u,
\end{equation}
where $\Gamma(X_\sigma, \mathcal{E})_u$ denotes the weight space corresponding to the character $u \in M$. Because $x_0$ lies in the dense torus orbit, every weight section in $\Gamma(X_\sigma, \mathcal{E})_u$ is uniquely determined by its value at $x_0$. Indeed, if $s\in\Gamma(X_\sigma,\mathcal E)_u$, then
\[
  s(tx_0)=\chi^{-u}(t)\,t\cdot s(x_0)
  \qquad (t\in T).
\]
Thus a weight section vanishing at $x_0$ vanishes on the
dense open orbit and hence everywhere. Evaluation at $x_0$
is therefore injective on each weight space. Restricting sections to the fiber $E = \mathcal{E}_{x_0}$ thus yields an injective evaluation map:
\begin{equation}
\label{eq:evaluation_embedding}
    \mathrm{ev}_{x_0} \colon \Gamma(X_\sigma, \mathcal{E})_u \hookrightarrow E.
\end{equation}
We denote the image of $\Gamma(X_\sigma, \mathcal{E})_u$ under $\mathrm{ev}_{x_0}$ in $E$ by 

\begin{equation}
\label{eq:evaluation_embedding_image}
E_u^\sigma:=\mathrm{ev}_{x_0}(\Gamma(X_\sigma, \mathcal{E})_u) \subseteq E.
\end{equation}

For any character $u' \in \sigma^\vee \cap M$, the character function $\chi^{u'}$ is regular on $X_\sigma$. Multiplication by $\chi^{u'}$ induces an injection of section spaces $\Gamma(X_\sigma, \mathcal{E})_u \hookrightarrow \Gamma(X_\sigma, \mathcal{E})_{u - u'}$. Since character multiplication commutes with evaluation at $x_0$, this induces a subspace inclusion in the fiber:
\begin{equation}
\label{eq:subspace_inclusion}
 E^\sigma_u\subseteq E^\sigma_{u-u'}
  \qquad (u'\in\sigma^\vee\cap M).
\end{equation}

If $u' \in \sigma^\perp \cap M$, the character $\chi^{u'}$ is invertible on $X_\sigma$, and the inclusion becomes an equality $E_u^\sigma = E_{u + u'}^\sigma$. Consequently, the subspace $E_u^\sigma$ depends only on the residue class $[u] \in M_\sigma \coloneqq M / (\sigma^\perp \cap M)$.

For every ray $\rho \in \Sigma(1)$ with primitive generator $v_\rho \in N$, the quotient $M_\rho = M / (\rho^\perp \cap M) \cong \mathbb{Z}$ is naturally identified with $\mathbb{Z}$ via the evaluation pairing $\langle \cdot, v_\rho \rangle$. This yields a decreasing $\mathbb{Z}$-filtration on $E$:
\begin{equation}
\label{eq:ray_filtration}
    \cdots \supseteq E^\rho(i-1) \supseteq E^\rho(i) \supseteq E^\rho(i+1) \supseteq \cdots,
\end{equation}
where $E^\rho(i) \coloneqq E_u^\rho$ for any $u \in M$ with $\langle u, v_\rho \rangle = i$. The filtration is bounded, satisfying
\begin{equation}
\label{eq:filtration_bounds}
    E^\rho(i) = E \quad \text{for } i \ll 0, \qquad \text{and} \qquad E^\rho(i) = 0 \quad \text{for } i \gg 0.
\end{equation}

The required compatibility condition is as follows. For every cone $\sigma\in\Sigma$, put $M_\sigma=M/(\sigma^\perp\cap M)$. There is a decomposition with finite support
\[
  E=\bigoplus_{w\in M_\sigma}E_{\sigma,w}
\]
such that, for every $\rho\in\sigma(1)$,
\[
  E^\rho(i)=
  \bigoplus_{\substack{w\in M_\sigma\\
                      \langle w,v_\rho\rangle\ge i}}
  E_{\sigma,w}.
\]
Equivalently, for $u\in M$,
\[
  E^\sigma_u=
  \bigoplus_{\substack{w\in M_\sigma\\
    \langle w-[u],v_\rho\rangle\ge0\ \text{for all }\rho\in\sigma(1)}}
  E_{\sigma,w}.
\]
The pairings are well defined because every element of
$\sigma^\perp\cap M$ vanishes on the linear span of $\sigma$.

Thus, on each cone, all of the filtrations corresponding to its rays admit a simultaneous splitting into weight spaces.

Klyachko's theorem states that the category of toric vector bundles on $X_\Sigma$ is naturally equivalent to the category whose objects are finite-dimensional vector spaces $E$ equipped with collections of decreasing filtrations

$$
\bigl\{E^\rho(i)\bigr\}_{\rho\in\Sigma(1),\,i\in\mathbb Z}
$$

satisfying this compatibility condition. A morphism of Klyachko data is a linear map

$$
f:E\longrightarrow E'
$$

such that

$$
f\bigl(E^\rho(i)\bigr)\subseteq E'^\rho(i)
$$

for every $\rho\in\Sigma(1)$ and every $i\in\mathbb Z$. In particular, two toric vector bundles are equivariantly isomorphic if and only if their Klyachko data are isomorphic.

\begin{example}
    For a line bundle $\mathcal L$, the fiber $L=\mathcal L_{x_0}$ is one-dimensional, and each filtration has the form

$$
L^\rho(i)
=
\begin{cases}
L,& i\leq a_\rho,\\
0,& i>a_\rho,
\end{cases}
$$

for integers $a_\rho$. Thus a toric line bundle is precisely the rank-one case of Klyachko's classification: its filtration data are given by one integer attached to each ray, subject to the usual Cartier compatibility condition.
\end{example}

\subsection{Splitting Criterion}

We first record a stopping criterion for the restriction argument developed below. It identifies when every ray filtration has a single jump, forcing the bundle to decompose as a direct sum of isomorphic toric line bundles.

For each ray $\rho \in \Sigma(1)$, let $v_\rho \in N$ denote its primitive lattice generator, and let $\lambda_\rho \colon \mathbb{G}_m \to T$ be the associated cocharacter. We write $S_\rho := \lambda_\rho(\mathbb{G}_m) \subset T$ for the corresponding $1$-parameter subtorus, and $D_\rho \subset \mathbb{P}^d$ for its associated torus-invariant divisor.

\begin{lem}
\label{lem:scalar-boundary}
Let $\mathcal{E}$ be a toric vector bundle of rank $r \ge 1$ on $\mathbb{P}^d$ with $d \ge 1$. If for every ray $\rho \in \Sigma(1)$, the subtorus $S_\rho$ acts on the restriction $\mathcal{E}|_{D_\rho}$ via a single character, then $\mathcal{E}$ splits equivariantly as

\begin{equation}
    \label{eq:bundle_iso}
    \mathcal E\simeq\mathcal L^{\oplus r}
\end{equation}
  
for some toric line bundle $\mathcal{L}$ on $\mathbb{P}^d$.
\end{lem}

\begin{proof}
Let $E = \mathcal{E}_{x_0}$ denote the fiber over the identity point $x_0 \in T$. For a maximal cone $\sigma \in \Sigma$, \cite[Proposition~1.1]{KM} gives a local $T$-equivariant decomposition over the affine chart $X_\sigma$:
\[
  \mathcal{E}|_{X_\sigma} \cong \bigoplus_{u \in M} \bigl( \mathcal{O}_{X_\sigma} \otimes E_{\sigma, u} \bigr),
\]
where $t \in T$ acts on the weight space $E_{\sigma, u}$ via the character $\chi^u(t)$.

For each ray $\rho \in \sigma(1)$, the one-parameter subtorus $S_\rho$ acts on $E_{\sigma, u}$ with weight $\langle u, v_\rho \rangle$. Under Klyachko's classification, the corresponding ray filtration on $E$ is given by
\[
  E^\rho(i) = \bigoplus_{\langle u, v_\rho \rangle \ge i} E_{\sigma, u}.
\]

By hypothesis, $S_\rho$ acts on $\mathcal{E}|_{D_\rho}$ via the single character $t \mapsto t^{a_\rho}$ for some $a_\rho \in \mathbb{Z}$. Consequently, every nonzero weight space $E_{\sigma, u}$ must satisfy $\langle u, v_\rho \rangle = a_\rho$, which implies that the ray filtration $E^\rho(i)$ has a single jump of full dimension at $i = a_\rho$:
\[
  E^\rho(i) = 
  \begin{cases}
    E, & i \le a_\rho, \\
    0, & i > a_\rho.
  \end{cases}
\]

Because the fan $\Sigma$ is smooth, the primitive ray generators of each maximal cone $\sigma \in \Sigma$ form a lattice basis for $N$. Thus, there exists a unique character $u_\sigma \in M$ such that $\langle u_\sigma, v_\rho \rangle = a_\rho$ for all $\rho \in \sigma(1)$. On a common face $\tau = \sigma \cap \sigma'$, these characters agree on $\operatorname{span}(\tau) \cap N$, so the collection $\{u_\sigma\}_{\sigma \in \Sigma}$ defines a toric line bundle $\mathcal{L}$.

The ray filtrations of $\mathcal{L}^{\oplus r}$ are identical to those of $\mathcal{E}$.  Hence $\mathcal E$ and $\mathcal L^{\oplus r}$ define isomorphic objects in the category of vector spaces with compatible filtrations. By the equivalence of categories, \eqref{eq:bundle_iso} holds for the corresponding toric vector bundles. In particular, $\mathcal E$ splits equivariantly as a direct sum of $r$ copies of a single toric line bundle.
\end{proof}

\subsection{Geometric Proof of Theorem \ref{thm:parabolics}}

\begin{proof}[Proof of Theorem \ref{thm:parabolics}]
Only the reverse implication requires proof.
We argue by induction on \(|A|\).
The assertion is immediate when \(|A|\leq m+1\).
Suppose therefore that
\[
|A|=n+1\geq m+2.
\]
Every proper subfamily satisfies the hypothesis, so the induction
assumption gives a common maximal torus for each proper subfamily.

For each \(\alpha\in A\), let
\begin{equation}
    \label{eq:filtration_for_ray}
    \mathcal F^\alpha:\qquad
0=\mathcal F^\alpha_0
\subsetneq\mathcal F^\alpha_1
\subsetneq\cdots
\subsetneq\mathcal F^\alpha_{s_\alpha}=V
\end{equation}

be the partial flag stabilized by \(P_\alpha\).
A collection of these parabolics contains a common maximal torus if and only if the corresponding flags in \eqref{eq:filtration_for_ray} admit a common coordinate
basis.

For each \(\alpha\), choose integers
\[
c^\alpha_1>c^\alpha_2>\cdots>c^\alpha_{s_\alpha}
\]
and define a decreasing \(\mathbb Z\)-filtration of \(V\) by
\[
F^\alpha(i)
\coloneqq
\mathcal F^\alpha_{r(i)},
\qquad
r(i)\coloneqq
\max\bigl(\{j\mid c^\alpha_j\geq i\}\cup\{0\}\bigr).
\]
Thus \(F^\alpha(i)=V\) for \(i\ll0\), \(F^\alpha(i)=0\)
for \(i\gg0\), and the nontrivial members of the filtration are
precisely the subspaces in \(\mathcal F^\alpha\).
The integers \(c^\alpha_j\) merely label the jumps; their particular
values do not affect the parabolic subgroup stabilizing the filtration. Associate these filtrations with the \(n+1\) rays of the fan of
\(\mathbb P^n\), whose effective torus is
\[
T_n=(\Bbbk^\times)^{n+1}/\Bbbk^\times,
\]
with diagonal denominator. Every $n$-dimensional cone corresponds to an affine chart of $\mathbb P^n$. It contains all but one of the rays, so its filtrations admit a common coordinate basis by induction. Since the primitive generators of such a cone form a lattice
basis, the chosen integral jumps determine characters on that cone. Thus Klyachko's compatibility condition is satisfied, and the filtrations define a rank-\(m\) toric vector bundle \(\mathcal E\) on \(\mathbb P^n\)
\cite[Theorem~2.2.1]{Klyachko}.
An equivariant splitting of $\mathcal E$ is equivalent to a
decomposition
\[
  V=\ell_1\oplus\cdots\oplus\ell_m,
  \qquad \dim\ell_i=1,
\]
that splits all the filtrations simultaneously.

Starting with \(X_0=\mathbb P^n\) and
\(\mathcal E_{0,1}=\mathcal E\), we construct coordinate linear
subspaces
\[
X_k\simeq\mathbb P^{n-k}
\]
and decompositions
\[
\mathcal E|_{X_k}
=\bigoplus_{j=1}^{q_k}\mathcal E_{k,j},
\qquad
q_0=1,
\]
where the summands are nonzero toric vector bundles for an
effective torus \(T_k\) on \(X_k\).

If the subtorus $S_D \subset T_k$ acts via a single character on each summand $\mathcal{E}_{k,j}|_D$ for every coordinate hyperplane $D \subset X_k$, the splitting criterion (Lemma~\ref{lem:scalar-boundary}) implies that every $\mathcal{E}_{k,j}$ splits into toric line bundles, verifying the statement.

Otherwise, choose a coordinate hyperplane $D \subset X_k$ for which the $S_D$-action on at least one summand $\mathcal{E}_{k,j}|_D$ exhibits two or more distinct weights, and set $X_{k+1} := D$. Because $S_D \cong \mathbb{G}_m$ fixes $D$ pointwise, its action on $\mathcal{E}_{k,j}|_D$ induces a decomposition into weight subbundles:
\[
  \mathcal{E}_{k,j}|_D = \bigoplus_{a \in \mathbb{Z}} \bigl( \mathcal{E}_{k,j}|_D \bigr)_a.
\]
Since $S_D$ is diagonalizable, these weight spaces form vector subbundles. 

Now select a complementary subtorus $T_{k+1} \subset T_k$ such that $T_k \cong S_D \times T_{k+1}$. The quotient torus $T_{k+1}$ acts effectively on $X_{k+1} \cong \mathbb{P}^{n-k-1}$. Moreover, because $T_k$ is abelian, $T_{k+1}$ commutes with $S_D$ and thus preserves each weight subbundle $(\mathcal{E}_{k,j}|_D)_a$. Collecting all nonzero weight subbundles defines the refined collection of summands $\{\mathcal{E}_{k+1, m}\}$ on $X_{k+1}$. Because at least one $\mathcal{E}_{k,j}|_D$ splits into two or more non-trivial weight subbundles, the total number of summands strictly increases:
\[
  q_{k+1} \ge q_k + 1.
\]

The total rank remains \(m\), so \(q_k\leq m\).
Consequently, after at most \(m-1\) restrictions the process
stops: either the splitting criterion applies earlier, or
there are \(m\) nonzero summands and all of them are line bundles.
For the resulting \(k\leq m-1\), the restriction
\(\mathcal E|_{X_k}\) splits and
\[
\dim (X_k)=n-k\geq n-m+1\geq 2.
\]
Hence \(\mathcal E\) restricts to a sum of line bundles on a coordinate plane \(\mathbb P^2\subseteq X_k\).
By Horrocks' splitting criterion, the bundle $\mathcal{E}$ splits on $\mathbb{P}^n$ as a sum of line bundles (\cite[Chapter~I, Theorem~2.3.2]{OSS}). Consequently, Klyachko's equivariant splitting criterion \cite[Theorem~1.2.3(ii)]{Klyachko} guarantees that $\mathcal{E}$ splits $T$-equivariantly.

The corresponding decomposition
\[
V=\ell_1\oplus\cdots\oplus\ell_m,
\qquad \dim\ell_i=1,
\]
splits every flag \(\mathcal F^\alpha\).
Choosing a nonzero vector in each \(\ell_i\) gives a common coordinate basis. Its diagonal torus is contained in every \(P_\alpha\), as required.
\end{proof}

\begin{rmk}
\label{rmk:gl_centralizers}
The restriction argument has a direct interpretation in terms
of structure groups.
At the \(k\)-th stage, write
\(r_{k,j}=\operatorname{rank}(\mathcal E_{k,j})\).
The decomposition reduces the frame bundle to
\[
H_k=\prod_{j=1}^{q_k}\operatorname{GL}_{r_{k,j}},
\qquad
\operatorname{rank}(H_k^{\mathrm{der}})
=\sum_{j=1}^{q_k}(r_{k,j}-1)=m-q_k,
\]
where
$H_k^{\mathrm{der}} = \prod\limits_{j=1}^{q_k} \operatorname{SL}_{r_{k,j}}$
is the derived subgroup of $H_k$.
Restricting to a coordinate hyperplane and taking weight
subbundles replaces \(H_k\) by the centralizer of the
cocharacter describing the stabilizer action on a fiber.
Every nontrivial refinement strictly decreases the rank of
the derived group. The splitting criterion is precisely the
case in which all the boundary cocharacters are central.
Thus the bound of \(m-1\) restrictions is the initial
semisimple rank
\(\operatorname{rank}(\operatorname{SL}_m)=m-1\).
This is the form of the argument that extends to other connected reductive structure groups.
\end{rmk}

\section{Helly's Theorem for Connected Reductive Groups}
\label{sec:reductive_groups}

Throughout this section, $G$ is a connected reductive complex algebraic group. We extend Theorem~\ref{thm:parabolics} to this setting. We write $n+1$ for the number of rays (and corresponding parabolic subgroups) and use the standard toric structure on $\mathbb{P}^n$.

\begin{definition}
A \emph{toric principal $G$-bundle} on $X_\Sigma$ is a principal right $G$-bundle $\mathcal{P} \to X_\Sigma$ equipped with a lift of the acting torus action that commutes with the right $G$-action. A principal bundle is called \emph{split} if it admits a reduction of structure group to a maximal torus. For a toric principal bundle, an \emph{equivariant splitting} is such a reduction preserved by the acting torus.    
\end{definition}

For the identity point $x_0$ of the open torus orbit, a framing of a toric principal $G$-bundle $\mathcal{P}$ is a choice of $p_0 \in \mathcal{P}_{x_0}$. For a cocharacter $\lambda \colon \mathbb{G}_m \to G$, let
\begin{equation}
    \label{eq:parabolic_lambda}
      P(\lambda) := \left\{ g \in G \mid \lim_{z \to 0} \lambda(z) g \lambda(z)^{-1} \text{ exists in } G \right\}
\end{equation}

be its associated parabolic subgroup.

Choose local equivariant trivializations $\mathcal{P}|_{X_\sigma} \cong X_\sigma \times G$ sending $p_0$ to $(x_0, 1)$, with torus action $t \cdot (x, h) = (tx, \rho_\sigma(t)h)$. For a ray $\rho \in \Sigma(1)$, let $v_\rho \colon \mathbb{G}_m \to T_n$ be its primitive one-parameter subgroup. Define its ray parabolic by
\begin{equation}
    \label{eq:ray_parabolic}
  P_\rho := P(\rho_\sigma \circ v_\rho), \qquad \rho \subseteq \sigma.
\end{equation}
This subgroup is independent of the cone $\sigma$ and the normalized trivialization \cite[Theorem~2.2]{KM}.

We first establish the realization, restriction, and splitting criteria needed for the two Helly-type theorems.


\begin{lem}
\label{lem:toric-construction}
Let $n\geq1$, and let $P_0,\ldots,P_n$ be parabolic subgroups of $G$.
Suppose that every proper subfamily contains a common maximal
torus. There is a framed toric principal $G$-bundle $\mathcal P$
on $\mathbb P^n$ whose ray parabolics are $P_0,\ldots,P_n$.
Moreover,
\[
  \mathcal P\text{ splits equivariantly}
  \quad\Longleftrightarrow\quad
  \bigcap_{\alpha=0}^{n}P_\alpha
  \text{ contains a maximal torus of }G.
\]
\end{lem}

\begin{proof}
Choose a faithful closed embedding $G \hookrightarrow \operatorname{GL}(W)$, where $W$ is a finite-dimensional vector space. For a cocharacter $\lambda \colon \mathbb{G}_m \to G$, let
\(
  P(\lambda)
\)
denote the corresponding parabolic subgroup of $G$ (see \eqref{eq:parabolic_lambda}). For each $0 \le \alpha \le n$, choose a cocharacter $\lambda_\alpha \colon \mathbb{G}_m \to G$ such that $P(\lambda_\alpha) = P_\alpha$, and let
\begin{equation}
    \label{eq:filtrations}
     F_\alpha(i) = \bigoplus_{a \ge i} W_a^{\lambda_\alpha}
\end{equation}
 be the associated decreasing weight filtration on $W$, where $W_a^{\lambda_\alpha} = \{ w \in W \mid \lambda_\alpha(t) \cdot w = t^a w \}$ is the weight space of weight $a$ under the $\mathbb{G}_m$-action induced by $\lambda_\alpha$.
For each maximal cone $\sigma$ of the fan of $\mathbb P^n$, choose a maximal torus $A_\sigma$ contained in the intersection of the parabolics corresponding to its rays. Since maximal tori of $P_\alpha$ are conjugate under
$P_\alpha$, some $p_\alpha\in P_\alpha$ conjugates
$\lambda_\alpha$ into $A_\sigma$. The subgroup
$P_\alpha$ stabilizes the weight filtration of
$\lambda_\alpha$ on $W$, so this conjugation leaves that
filtration  $F_\alpha(\bullet)$ unchanged. The resulting commuting cocharacters give a simultaneous splitting of the filtrations belonging to $\sigma$. Since its primitive ray generators form a lattice basis, they determine a homomorphism
\[
  \rho_\sigma:T_n\longrightarrow A_\sigma\subseteq G.
\]
By Klyachko's classification theorem, these ray filtrations determine  a toric vector bundle $\mathcal E$ with fiber $W$ and the filtrations in \eqref{eq:filtrations}.

Using local equivariant trivializations normalized by identifying the fiber $\mathcal{E}_{x_0}$ with $W$, the transition functions on the open orbit are
\[
  g_{\tau\sigma}(t) = \rho_\tau(t) \rho_\sigma(t)^{-1} \in G.
\]
By construction, these transition functions extend regularly to $\operatorname{GL}(W)$-valued functions on overlaps $X_\tau \cap X_\sigma$. Since $G$ is closed in $\operatorname{GL}(W)$ and the open orbit is dense in each overlap, the extended functions take values in $G$. Gluing trivial principal $G$-bundles over the affine cover yields a toric principal $G$-bundle $\mathcal{P}$. Since $g_{\tau\sigma}(x_0) = 1$, the points $(x_0, 1)$ define a common point $p_0 \in \mathcal{P}_{x_0}$, giving the required framing.

By \cite[Theorem~2.2]{KM}, $\mathcal{P}$ has the prescribed cocharacter class on each ray. The equivariant splitting criterion \cite[Corollary~5.5]{HK} asserts that $\mathcal{P}$ splits equivariantly if and only if these ray classes admit representatives in a single maximal torus $A \subset G$. A class represented by $\lambda_\alpha$ admits a representative in $A$ if and only if $A \subseteq P_\alpha$. Consequently, all classes have representatives in $A$ if and only if $A \subseteq \bigcap\limits_{\alpha=0}^n P_\alpha$.

\end{proof}

\begin{lem}
\label{lem:centralizer_restriction}
Let $\mathcal{Q}$ be a toric principal $G$-bundle on $\mathbb{P}^d$ with $d \ge 2$. Let $D \subset \mathbb{P}^d$ be a coordinate hyperplane, and let $S \cong \mathbb{G}_m$ be the $1$-parameter subtorus of $T_d$ that fixes $D$ pointwise. The action of $S$ on the fibers of $\mathcal{Q}|_D$ determines a cocharacter $\lambda \colon S \to G$ up to $G$-conjugacy, whose conjugacy class is independent of the point chosen on $D$. The following statements hold:
\begin{enumerate}
    \item The restriction $\mathcal{Q}|_D$ admits a $T_d$-invariant reduction of structure group to the Levi subgroup
    \[
        L = C_G(\lambda(S)).
    \]
    For any complementary torus $T' \subset T_d$ such that $T_d \cong S \times T'$, this reduction defines a toric principal $L$-bundle on $D \cong \mathbb{P}^{d-1}$ with respect to the action of $T'$. If $\lambda(S) \not\subset Z(G)$, then the semisimple rank satisfies
    \begin{equation}
        \label{eq:rank_inequality}
        \operatorname{rk}(L_{\mathrm{der}}) \le \operatorname{rk}(G_{\mathrm{der}}) - 1.
    \end{equation}

    \item If the fiber cocharacters associated with all coordinate hyperplanes take values in the center $Z(G)$, then $\mathcal{Q}$ splits equivariantly.
\end{enumerate}
\end{lem}

\begin{proof}
By local equivariant triviality \cite[Theorem~2.5]{KM}, $\mathcal{Q}$ admits local trivializations over the standard affine toric charts $X_\sigma \subset \mathbb{P}^d$ in which the $T_d$-action is given by $t \cdot (x, h) = (tx, \rho_\sigma(t)h)$ for homomorphisms $\rho_\sigma \colon T_d \to G$. Over any chart meeting $D$, the fiber action of $S$ is given by the restricted cocharacter $\rho_\sigma|_S \colon S \to G$. Transition functions conjugate these homomorphisms on overlaps. Since the chart intersections with $D$ are connected and mutually overlap, all such restrictions are $G$-conjugate. Thus, after conjugating the local trivializations by constant elements of $G$, we may assume that $\rho_\sigma|_S = \lambda$ for a fixed cocharacter $\lambda \colon S \to G$ across all charts meeting $D$.

Define the subbundle
\[
\mathcal Q_L
:=
\{q\in\mathcal Q|_D\mid
s\cdot q=q\lambda(s)\text{ for every }s\in S\}.
\]
In the normalized local trivializations over $X_\sigma \cap D$, the subbundle $\mathcal{Q}_L$ is $(X_\sigma \cap D) \times L$, where $L = C_G(\lambda(S))$. Hence, $\mathcal{Q}_L$ is a principal $L$-subbundle of $\mathcal{Q}|_D$, giving a reduction of structure group to $L$. Because the $T_d$-action on $\mathcal{Q}$ commutes with both $S$ and the right $G$-action, it preserves $\mathcal{Q}_L$. The short exact sequence of algebraic tori
\[
    1 \longrightarrow S \longrightarrow T_d \longrightarrow T_d/S \longrightarrow 1
\]
splits. Choosing a splitting identifies $T_d/S$ with a complementary subtorus $T' \subset T_d$. The torus $T' \cong T_{d-1}$ acts effectively on $D \cong \mathbb{P}^{d-1}$, making $\mathcal{Q}_L$ a toric principal $L$-bundle over $D$.

The centralizer $L$ is connected and reductive.
To prove the rank inequality \eqref{eq:rank_inequality}, choose a maximal torus of $G$
containing $\lambda(S)$, and let $\Phi_G$ be its root system.
The root system of $L$ is
\[
\Phi_L=\{\alpha\in\Phi_G\mid
\langle\alpha,\lambda\rangle=0\}.
\]
If $\lambda(S)$ is noncentral, the functional
$\alpha\mapsto\langle\alpha,\lambda\rangle$ is nonzero on
the span of $\Phi_G$. Thus the span of $\Phi_L$ has dimension
at most $\operatorname{rank}(G_{\mathrm{der}})-1$.

For the second assertion, assume the fiber cocharacters associated with all coordinate hyperplanes $D_0, \dots, D_d$ take values in $Z(G)$. Choose local trivializations normalized by a common frame over the identity point $x_0$ of the dense orbit $T_d$. For each maximal cone $\sigma$, the primitive generators of its rays correspond to the $1$-parameter subtori fixing the coordinate hyperplanes meeting $X_\sigma$. Because these primitive ray generators form a lattice basis of $N$, and the restricted homomorphisms $\rho_\sigma|_S$ take values in $Z(G)$, it follows that 
\[
    \rho_\sigma(T_d) \subseteq Z(G)^0
\]
for every maximal cone $\sigma$. On the open orbit, the transition function between two such trivializations is $g_{\tau\sigma}(t) = \rho_\tau(t)\rho_\sigma(t)^{-1} \in Z(G)^0$. Since $Z(G)^0$ is a closed subgroup of $G$, these transition functions take values in $Z(G)^0$ over all overlaps. The local copies of $Z(G)^0$ therefore glue to yield an equivariant reduction of structure group to $Z(G)^0$, and subsequently to any maximal torus of $G$ containing $Z(G)^0$. Thus, $\mathcal{Q}$ splits equivariantly.
\end{proof}

\begin{lem}
\label{lem:principal_horrocks}
Let $\mathcal{Q}$ be a toric principal $G$-bundle on $\mathbb{P}^d$. If the restriction of $\mathcal{Q}$ to a linear subspace $Y \cong \mathbb{P}^q\subset \mathbb{P}^d$ ($2\le q \le d$) splits, then $\mathcal{Q}$ splits equivariantly.
\end{lem}

\begin{proof}
Fix a faithful closed embedding $\iota \colon G \hookrightarrow \mathrm{GL}(V)$, and let $E = \mathcal Q \times_G V$ be the associated toric vector bundle on $\mathbb P^d$. Since $\mathcal Q|_Y$ splits, its reduction to a maximal torus of $G$ induces a decomposition of $E|_Y$ into a direct sum of line bundles. As $Y \cong \mathbb{P}^q$ and $q \ge 2$, Horrocks' splitting criterion \cite[Chapter~I, Theorem~2.3.2]{OSS} implies that $E$ splits on $\mathbb P^d$.

Next, \cite[Theorem~1.2.3(ii)]{Klyachko} guarantees that $E$ splits equivariantly. It then follows from \cite[Theorem~6.9]{DKBDP} that $\mathcal Q$ itself splits equivariantly.
\end{proof}




We now extend Theorem~\ref{thm:parabolics} to parabolic subgroups of connected reductive algebraic groups over $\mathbb{C}$, establishing a Helly threshold governed by the semisimple rank of $G$.

\begin{theorem}
\label{thm:parabolicsGen}
Let $m=\operatorname{rank}(G^{\mathrm{der}})$ be the semisimple rank of $G$. Let $\{P_\alpha\}_{\alpha\in A}$ be a finite family of parabolic subgroups of $G$. The intersection $\underset{\alpha\in A}{\bigcap}P_\alpha$ contains a common maximal torus of $G$ if and only if the intersection of every subfamily of at most $m+2$ members contains a common maximal torus of $G$.
\end{theorem}

\begin{proof}
We proceed by induction on $|A|$. The base case $|A| \le m+2$ holds trivially. Assume $|A| = n+1 \ge m+3$, and suppose the theorem holds for all smaller families. 

By the inductive hypothesis, every proper subfamily of $\{P_\alpha\}_{\alpha \in A}$ contains a common maximal torus. Indexing the family as $\{P_0, \dots, P_n\}$, Lemma~\ref{lem:toric-construction} yields a framed toric principal $G$-bundle $\mathcal{P}$ on $\mathbb{P}^n$ whose ray parabolics are $P_0, \dots, P_n$. By Lemma~\ref{lem:toric-construction}, the intersection $\bigcap\limits_{\alpha=0}^n P_\alpha$ contains a common maximal torus if and only if $\mathcal{P}$ splits equivariantly.

To establish equivariant splitting, we construct a sequence of linear coordinate subspaces $X_k \cong \mathbb{P}^{n-k} \subset \mathbb{P}^n$, connected reductive subgroups $H_k \subseteq G$, and toric principal $H_k$-bundles $\mathcal{P}_k$ on $X_k$, initialized by
\[
    (X_0, H_0, \mathcal{P}_0) = (\mathbb{P}^n, G, \mathcal{P}).
\]
At stage $k$, the bundle $\mathcal{P}_k$ is a reduction of structure group of $\mathcal{P}|_{X_k}$ to $H_k$.

Suppose the construction has reached stage $k$. If all fiber cocharacters of $\mathcal{P}_k$ associated with the coordinate hyperplanes of $X_k$ take values in $Z(H_k)$, then by Lemma~\ref{lem:centralizer_restriction}(2), $\mathcal{P}_k$ splits equivariantly, and the construction terminates at step $k$. Otherwise, choose a coordinate hyperplane $X_{k+1} \subset X_k$ whose associated fiber cocharacter $\lambda_k$ is noncentral in $H_k$, and set
\[
    H_{k+1} = C_{H_k}(\lambda_k(\mathbb{G}_m)).
\]
By Lemma~\ref{lem:centralizer_restriction}(1), the restriction $\mathcal{P}_k|_{X_{k+1}}$ admits a reduction of structure group to $H_{k+1}$, giving rise to a toric principal $H_{k+1}$-bundle $\mathcal{P}_{k+1}$ on $X_{k+1} \cong \mathbb{P}^{n-(k+1)}$, where toricity is defined with respect to a complementary torus $T' \subset T_n$, which acts effectively on $X_{k+1}$.

Let $r_k = \operatorname{rank}(H_k^{\mathrm{der}})$ denote the semisimple rank of $H_k$, so that $r_0 = m$. As $\lambda_k$ is noncentral in $H_k$, Lemma~\ref{lem:centralizer_restriction}(1) implies that $r_{k+1} \le r_k - 1$. Consequently, this process must terminate after at most $k \le m$ steps, since $r_m = 0$ implies that $H_m$ is a torus, in which case $\mathcal{P}_m$ is automatically split.

At the termination step $k \le m$, the dimension of $X_k$ satisfies
\[
    \dim X_k = n - k \ge (m + 2) - m = 2.
\]
Because $\mathcal{P}_k$ splits equivariantly, the bundle $\mathcal{P}|_{X_k}$ admits a reduction of structure group to a torus in $G$, and hence splits. Since $\dim (X_k) \ge 2$, Lemma~\ref{lem:principal_horrocks} implies that $\mathcal{P}$ splits equivariantly on all of $\mathbb{P}^n$. By Lemma~\ref{lem:toric-construction}, it follows that $\bigcap\limits_{\alpha=0}^n P_\alpha$ contains a common maximal torus of $G$.
\end{proof}

\begin{rmk}
\label{rmk:gl_specialization}
For $G = \operatorname{GL}(W)$, where $W$ is an $m$-dimensional complex vector space, the derived subgroup is $G^{\mathrm{der}} = \operatorname{SL}(W)$ with semisimple rank $\operatorname{rk}(G_{\mathrm{der}}) = m - 1$. The Helly threshold in Theorem~\ref{thm:parabolicsGen} thus becomes $(m - 1) + 2 = m + 1$, recovering Klyachko's original bound in Theorem~\ref{thm:klyachko-helly-flags}.
\end{rmk}

\begin{rmk}
\label{rmk:classical_groups_flags}
For other classical groups, such as the symplectic group $\operatorname{Sp}(2m, \mathbb{C})$ or the orthogonal groups $\operatorname{SO}(N, \mathbb{C})$ (with $N = 2m$ or $2m+1$ and $m \ge 2$), parabolic subgroups are geometrically realized as stabilizers of isotropic flags. Under this realization, Theorem~\ref{thm:parabolicsGen} yields a Helly-type property for isotropic flags: a finite collection of isotropic flags shares a common adapted basis compatible with the underlying non-degenerate bilinear form if and only if every subcollection of at most $m+2$ flags shares such a basis, where $m$ is the semisimple rank of $G$.
\end{rmk}

\section{Helly's Number for Connected Reductive Groups}
\label{subsec:helly-by-type}

Following \cite[Section~6]{HK}, we recall the notion of Helly's number for parabolic subgroups.

\begin{definition}
\label{defn:helly_number}
The \emph{Helly number} of $G$, denoted $h(G)$, is the least positive integer $h$ such that a finite family of parabolic subgroups of $G$ contains a common maximal torus whenever every subfamily of at most $h$ members contains a common maximal torus.
\end{definition}


We now examine the sharpness of Theorem~\ref{thm:parabolicsGen}. The first point is that the Helly number is insensitive to central isogenies and behaves as a maximum, rather than a sum, under direct products.

\begin{prop}
\label{prop:helly-central-products}
Let $G$ be a connected reductive complex algebraic group of positive
semisimple rank. Let
\[
  \widetilde{G}^{\mathrm{der}}
  =G_1\times\cdots\times G_s
\]
be the decomposition of the simply connected cover of
$G^{\mathrm{der}}$ into its almost simple factors. Then
\begin{equation}
\label{eq:helly-product-formula}
  h(G)=\max_{1\leq i\leq s}h(G_i).
\end{equation}
In particular, if $r_i=\operatorname{rank}(G_i)$, then
\begin{equation}
\label{eq:factorwise-rank-upper-bound}
  h(G)\leq 2+\max_{1\leq i\leq s}r_i.
\end{equation}
\end{prop}

\begin{proof}
A central isogeny induces bijections on parabolic subgroups and on maximal tori, preserving containment. Hence it does not change the Helly number. The connected center of a reductive group is contained in every parabolic subgroup and likewise does not affect whether an intersection contains a maximal torus. It therefore suffices to consider the direct product $G_1\times\cdots\times G_s$.

Every parabolic subgroup and every maximal torus of this product has the form
\[
  P_1\times\cdots\times P_s,
  \qquad
  T_1\times\cdots\times T_s,
\]
respectively. A family of parabolic subgroups of the product has a common maximal torus if and only if its projected family has a common maximal torus in every factor. This proves the upper bound
\[
h(G)\leq\max_i h(G_i).
\]

Conversely, a minimally incompatible family in one factor $G_i$ can be regarded as a family in the product by taking the whole group in all other factors. This proves the reverse inequality and hence \eqref{eq:helly-product-formula}. Applying Theorem~\ref{thm:parabolicsGen} to each $G_i$ gives \eqref{eq:factorwise-rank-upper-bound}.
\end{proof}

We will also use monotonicity with respect to closed reductive
subgroups.

\begin{lem}
\label{lem:helly-subgroup-monotonicity}
If $H\subseteq G$ are connected reductive complex algebraic groups,
then
\[
  h(H)\leq h(G).
\]
\end{lem}

\begin{proof}
If $H$ is a torus, then $h(H) = 1 \le h(G)$.
Suppose that $H$ has positive semisimple rank and set $q = h(H)$.
By definition, $H$ satisfies the $q$-Helly property but fails the $(q-1)$-Helly property. The failure of the $(q-1)$-Helly property yields an incompatible family of parabolic subgroups of $H$ all of whose subfamilies of size at most $q-1$ are compatible. An inclusion-minimal incompatible subfamily of this collection must have size at least $q$, while the $q$-Helly property forces its size to be at most $q$. Hence, its size is exactly $q$, giving a minimally incompatible family of $q$ parabolic subgroups.

By Lemma~\ref{lem:toric-construction}, this family is realized by a toric principal $H$-bundle $\mathcal E$ on $\mathbb P^{q-1}$ that does not split equivariantly.
Extend the structure group to obtain $\mathcal E_G = \mathcal E \times^H G$. By \cite[Corollary~5.6]{HK}, $\mathcal E_G$ does not split equivariantly.

Let $Q_1, \ldots, Q_q$ be its ray parabolics.
Every proper subset of the rays of the fan of $\mathbb P^{q-1}$ is contained in a cone. The local compatibility condition therefore implies that every proper subfamily of the $Q_i$ contains a common maximal torus of $G$. If all the $Q_i$ contained a common maximal torus, the equivariant splitting criterion would split $\mathcal E_G$, a contradiction. Thus, the $Q_i$ form a minimally incompatible family of size $q$, showing that $h(G) \ge q$.
\end{proof}

For classical groups of types $B_m$ and $C_m$, demonstrating the sharpness of the bounds on the respective Helly numbers requires a configuration distinct from the standard linearly dependent family of isotropic lines. Recall that an apartment in the building of a classical group is uniquely determined by a polar frame. Consequently, a family of isotropic subspaces lies in a common apartment if and only if all of its members are coordinate subspaces relative to a single polar frame.

\begin{prop}
\label{prop:helly-BC-sharp}
For every $m\geq1$,
\[
  h\bigl(\operatorname{Sp}_{2m}\bigr)=m+2,
  \qquad
  h\bigl(\operatorname{SO}_{2m+1}\bigr)=m+2.
\]
\end{prop}

\begin{proof}
The upper bounds follow from
Theorem~\ref{thm:parabolicsGen}. We construct minimally
incompatible families of $m+2$ parabolic subgroups.

First let $V=E\oplus F$ be a $2m$-dimensional symplectic vector
space with bases
\[
  E=\langle e_1,\ldots,e_m\rangle,
  \qquad
  F=\langle f_1,\ldots,f_m\rangle,
\]
such that
\[
  \omega(e_i,f_j)=\delta_{ij},
  \qquad
  \omega(e_i,e_j)=\omega(f_i,f_j)=0.
\]
Choose $a_1,\ldots,a_m\in\mathbb C^\times$ and put
\[
  f_a=\sum_{i=1}^m a_i f_i,
  \qquad
  \ell(v)\coloneqq\omega(v,f_a),
  \qquad
  E_a=\{v+\ell(v)f_a\mid v\in E\}.
\]
The subspace $E_a$ is Lagrangian. Consider the $m+2$ parabolic subgroups stabilizing the isotropic subspaces

\begin{equation}
\label{eq:symplectic-minimal-obstruction}
  \mathbb C \langle f_1\rangle,\ldots,\mathbb C \langle f_m\rangle,
  \qquad E,\qquad E_a.
\end{equation}

Every proper subfamily of \eqref{eq:symplectic-minimal-obstruction}
shares an apartment. It suffices to consider subfamilies obtained by deleting one member, since every smaller subfamily is contained in one of these.

If $E_a$ is omitted, use the standard frame
$(e_i,f_i)_{i\in\{1,\dots,m\}}$. If $E$ is omitted, use the frame
\[
  \bigl(e_i+a_i f_a,f_i\bigr)_{i\in\{1,\dots,m\}}.
\]
Finally, suppose that $\mathbb{C}\langle f_k\rangle$ is omitted. For each $i \neq k$, set
\[
  h_i \coloneqq e_i - \frac{a_i}{a_k} e_k.
\]
Then
\begin{equation}
\label{eq:symplectic-deleted-frame}
  (h_i, f_i)_{i \neq k}, \qquad \left(e_k, \frac{e_k + a_k f_a}{a_k^2}\right)
\end{equation}
forms a symplectic frame. In this frame, the previously defined subspaces satisfy
\[
  E=\operatorname{span}\bigl(\{h_i:i\ne k\}\cup\{e_k\}\bigr),
  \qquad
  E_a=\operatorname{span}\bigl(\{h_i:i\ne k\}
                    \cup\{e_k+a_kf_a\}\bigr).
\]
Then both $E$ and $E_a$ are coordinate Lagrangians for this frame. 

However, the entire family cannot lie in a common apartment. Indeed, any apartment containing all the isotropic lines $\mathbb{C}f_i$ (and hence their span $F$) together with the Lagrangian $E$ must be determined by the canonical frame $(e_i, f_i)_{i=1}^m$. Relative to this frame, $E$ is the unique coordinate Lagrangian that is transverse to $F$. However, $E_a$ is also transverse to $F$ and $E_a \neq E$, so $E_a$ cannot be a coordinate Lagrangian for this frame.

Consequently, no single apartment contains the entire family. Since every proper subfamily lies in an apartment, the collection in \eqref{eq:symplectic-minimal-obstruction} is minimally incompatible, establishing the equality
\[
  h(\operatorname{Sp}_{2m}) = m + 2.
\]

For the odd orthogonal group, let
\[
  V=E\oplus F\oplus\mathbb C z
\]
and equip $V$ with the symmetric bilinear form determined by
\[
  (e_i,f_j)=\delta_{ij},
  \qquad
  (z,z)=2,
\]
all other pairings between the displayed basis vectors being zero.
With $f_a$ and $\ell$ as above, define
\[
  E_a=
  \{x+\ell(x)z-\ell(x)f_a\mid x\in E\}.
\]
A direct computation shows that $E_a$ is maximal isotropic.
Again consider the stabilizers of
\[
  \mathbb C \langle f_1\rangle ,\ldots,\mathbb C \langle f_m\rangle ,
  \qquad E,\qquad E_a.
\]
If $E_a$ is omitted, use the standard orthogonal frame. If $E$ is
omitted, use the hyperbolic pairs
\[
  \bigl(e_i+a_i z-a_i f_a,f_i\bigr)_{i\in\{1,\dots,m\}}.
\]
Their orthogonal complement is a nonisotropic line. If
$\mathbb C \langle f_k\rangle$ is omitted, use
\[
  (h_i,f_i)_{i\neq k},
  \qquad
 \left(e_k,-\frac{e_k+a_kz-a_kf_a}{a_k^2}\right),
\]
together with the nonisotropic orthogonal complement. These are orthogonal frames containing all the remaining subspaces.

As in the symplectic case, all the lines $\mathbb C\langle f_i\rangle$ together with $E$ force the partner lines $\mathbb C e_i$, whereas the distinct maximal isotropic subspace $E_a$, which is transverse to $F$, is not coordinate for this frame. Hence the whole family is not contained in an apartment, and $h(\operatorname{SO}_{2m+1})= m+2$.
\end{proof}

Recall that Dynkin's classification of maximal subgroup inclusions yields the following maximal-rank subgroup inclusions (up to central isogeny):
\begin{equation}
\label{eq:maximal-rank-subgroups}
  A_2 \subset G_2,
  \qquad
  B_4 \subset F_4,
  \qquad
  A_5 \times A_1 \subset E_6,
  \qquad
  A_7 \subset E_7,
  \qquad
  A_8 \subset E_8;
\end{equation}
see, for example, \cite{Dyn57}. We can now combine the above subgroup inclusions and explicit construction to determine or bound the Helly numbers for all almost simple connected complex algebraic groups.

\begin{cor}
\label{cor:helly-values-by-type}
Let $G$ be an almost simple connected complex algebraic group of semisimple rank $m$. The exact values and bounds for the Helly number $h(G)$ across all Lie types are summarized in Table~\ref{tab:helly-values-by-type}.

\begin{table}[htbp]
\centering
\begingroup
\normalfont
\begin{tabular}{c|c|l}
\hline
\textup{Type of $G$} & \textup{Helly number $h(G)$} & \textup{Reason for the lower bound}\\
\hline
$A_m$ & $h(G) = m+2$ & Line configuration in $\mathbb C^{m+1}$\\
$B_m$ & $h(G) = m+2$ & Proposition~\ref{prop:helly-BC-sharp}\\
$C_m$ & $h(G) = m+2$ & Proposition~\ref{prop:helly-BC-sharp}\\
$D_m$, $m \ge 4$
  & $m+1 \le h(G) \le m+2$
  & Inclusion \(\mathrm{GL}_m\subset\mathrm{SO}_{2m}\)\\
$G_2$ & $h(G) = 4$ & Inclusion $A_2 \subset G_2$\\
$F_4$ & $h(G) = 6$ & Inclusion $B_4 \subset F_4$\\
$E_6$ & $7 \le h(G) \le 8$ & Inclusion $A_5 \times A_1 \subset E_6$\\
$E_7$ & $h(G) = 9$ & Inclusion $A_7 \subset E_7$\\
$E_8$ & $h(G) = 10$ & Inclusion $A_8 \subset E_8$\\
\hline
\end{tabular}
\endgroup
\caption{Helly numbers $h(G)$ for almost simple complex algebraic groups $G$ of semisimple rank $m$.}
\label{tab:helly-values-by-type}
\end{table}

Thus, the only remaining gaps for almost simple groups are
\begin{equation}
\label{eq:remaining-simple-gaps}
  h(D_m) \in \{m+1, m+2\} \quad (m \ge 4),
  \qquad
  h(E_6) \in \{7, 8\}.
\end{equation}
\end{cor}

\begin{proof}
Theorem~\ref{thm:parabolicsGen} establishes the upper bound $h(G) \le m+2$ for $G$ of type $A_m$. To prove the reverse inequality, consider the following $m+2$ lines in $\mathbb{C}^{m+1}$:
\[
  \mathbb{C}\langle e_1\rangle, \dots, \mathbb{C}\langle e_{m+1}\rangle, \quad \text{and} \quad \mathbb{C}\langle e_1 + \dots + e_{m+1}\rangle,
\]
where $\{e_1, \dots, e_{m+1}\}$ denotes the standard basis of $\mathbb{C}^{m+1}$. While any subcollection of $m+1$ lines is adapted to a common coordinate basis, the complete set of $m+2$ lines cannot serve as simultaneous coordinate lines for any single basis. Consequently, $h(G) = m+2$.

Because the algebraic groups $\operatorname{GL}_{m+1}$, $\operatorname{SL}_{m+1}$, and $\operatorname{PGL}_{m+1}$ share the same underlying spherical Tits building, their Helly numbers are identical:
\[
  h(\operatorname{GL}_{m+1}) = h(\operatorname{SL}_{m+1}) = h(\operatorname{PGL}_{m+1}) = m+2.
\]
Having settled type $A_m$, it remains to address the classical type $D_m$ ($m \ge 4$) and the exceptional Lie types.

The general upper bounds $h(G) \le m+2$ are supplied by Theorem~\ref{thm:parabolicsGen}. Lemma~\ref{lem:helly-subgroup-monotonicity}, the product formula \eqref{eq:helly-product-formula}, and the maximal-rank inclusions \eqref{eq:maximal-rank-subgroups} yield the lower bounds:
\[
\begin{aligned}
  h(G_2) &\ge h(A_2) = 4,\\
  h(F_4) &\ge h(B_4) = 6,\\
  h(E_6) &\ge h(A_5 \times A_1) = \max\{h(A_5), h(A_1)\} = \max\{7, 3\} = 7,\\
  h(E_7) &\ge h(A_7) = 9,\\
  h(E_8) &\ge h(A_8) = 10,
\end{aligned}
\]
where $h(X)$ denotes $h(G)$ for a group $G$ of type $X$.
Comparing these lower bounds with the rank upper bound $m+2$ determines every asserted exceptional value and restricts $h(E_6)$ to the set $\{7, 8\}$.

For groups of type $D_m$ ($m \ge 4$), consider the standard closed embedding
\[
  \operatorname{GL}_m \hookrightarrow \operatorname{SO}_{2m}, \qquad A \longmapsto \begin{pmatrix} A & 0 \\ 0 & (A^{-1})^{\mathsf{T}} \end{pmatrix},
\]
defined relative to the split symmetric form on $\mathbb{C}^m \oplus (\mathbb{C}^m)^*$. Applying Lemma~\ref{lem:helly-subgroup-monotonicity} to this embedding yields the lower bound
\[
  h(D_m) \ge h(\operatorname{GL}_m) = m + 1,
\]
while Theorem~\ref{thm:parabolicsGen} guarantees $h(D_m) \le m + 2$. 

For low-rank cases, note that $D_3 \cong A_3$, which gives $h(D_3) = 5$. Type $D_2 \cong A_1 \times A_1$ is reducible, with Helly number $h(D_2) = 3$ by Proposition~\ref{prop:helly-central-products}.
\end{proof}

For a general connected reductive group, the determination of the Helly number reduces directly to its almost simple factors.

\begin{cor}
\label{cor:factorwise-helly-classification}
Let $G$ be a connected reductive complex algebraic group of positive semisimple rank, and let $G_1, G_2, \dots, G_s$ be the almost simple factors of the simply connected cover of its derived subgroup $G^{\mathrm{der}}$. Then the Helly number of $G$ satisfies
\[
  h(G) = \max_{1 \le i \le s} h(G_i).
\]
In particular, if no factor $G_i$ is of type $D_m$ with $m \ge 4$ or of type $E_6$, then
\begin{equation}
\label{eq:exact-factorwise-helly}
  h(G) = 2 + \max_{1 \le i \le s} \operatorname{rank}(G_i).
\end{equation}

Consequently, any remaining ambiguity in $h(G)$ is inherited directly, factor by factor, from the unresolved cases in \eqref{eq:remaining-simple-gaps}.
\end{cor}

When $G$ is a torus, its unique parabolic subgroup is $G$ itself; hence, Definition~\ref{defn:helly_number} yields $h(G) = 1$.

\begin{rmk}
\label{rmk:total-rank-not-sharp}
The total semisimple rank does not uniquely determine $h(G)$ for a non-simple semisimple group. For example,
\[
  h\bigl((\operatorname{SL}_2)^s\bigr) = 3
\]
for all $s \ge 1$, despite $\operatorname{rank}\bigl((\operatorname{SL}_2)^s\bigr) = s$.
\end{rmk}

\section{Implications for Toric \texorpdfstring{$G$}{G}-bundles}
\label{sec:impl_for_toric_G-bundles}

Throughout this section, $G$ is a connected reductive complex algebraic group. All splitting statements for toric principal bundles mean equivariant splitting, that is, the existence of an equivariant reduction to a maximal torus of $G$.

We recall a splitting criterion for toric principal bundles, stemming from the classification of toric bundles \cite{KM} and the apartment splitting criterion \cite[Corollary~5.5]{HK}.

\begin{prop}
\label{prop:ray-parabolic-realization}
Let $X_\Sigma$ be a smooth toric variety. A collection of parabolic subgroups $\{P_\rho\}_{\rho \in \Sigma(1)}$ of $G$ arises as the ray parabolics of a framed toric principal $G$-bundle on $X_\Sigma$ if and only if
\begin{equation}
\label{eq:ray-parabolic-compatibility}
  \bigcap_{\rho \in \sigma(1)} P_\rho \text{ contains a maximal torus of } G \quad \text{for every cone } \sigma \in \Sigma.
\end{equation}
Moreover, such a bundle splits equivariantly if and only if all of its ray parabolics share a common maximal torus.
\end{prop}

\begin{proof}
To establish existence, fix a faithful closed embedding $G \hookrightarrow \operatorname{GL}(W)$. For each ray $\rho \in \Sigma(1)$, choose a cocharacter $\lambda_\rho \colon \mathbb{G}_m \to G$ such that $P(\lambda_\rho) = P_\rho$ (see \eqref{eq:parabolic_lambda}); if $P_\rho = G$, we take $\lambda_\rho = 0$.

By hypothesis \eqref{eq:ray-parabolic-compatibility}, for each cone $\sigma \in \Sigma$, there exists a maximal torus $A_\sigma \subset G$ contained in the intersection $\bigcap_{\rho \in \sigma(1)} P_\rho$. For each $\rho \in \sigma(1)$, since all maximal tori of $P_\rho$ are conjugate within $P_\rho$, there exists an element $p_\rho \in P_\rho$ conjugating $\lambda_\rho$ into $A_\sigma$. Because $P_\rho$ stabilizes the weight filtration $F_\rho(\bullet)$ on $W$ associated with $\lambda_\rho$, this conjugation leaves $F_\rho(\bullet)$ unchanged.

Since $X_\Sigma$ is smooth, the primitive ray generators of any cone $\sigma \in \Sigma$ form a subset of a lattice basis for cocharacter lattice $N$ (see \eqref{eq:cochar_lattice}). Complete the primitive ray generators of $\sigma$ to a lattice basis of $N$. Assigning the trivial cocharacter to the additional basis vectors extends the prescribed
cocharacters to a homomorphism
\[
  \rho_\sigma\colon T\longrightarrow A_\sigma\subseteq G.
\]
These homomorphisms guarantee that the collection of ray filtrations $\{F_\rho(\bullet)\}_{\rho \in \Sigma(1)}$ satisfies Klyachko's compatibility condition for toric vector bundles on each affine chart $X_\sigma$.

As in Lemma~\ref{lem:toric-construction}, normalized local trivializations over the affine cover $\{X_\sigma\}_{\sigma \in \Sigma}$ yield transition functions $g_{\tau\sigma}(t) = \rho_\tau(t)\rho_\sigma(t)^{-1} \in G$ on the dense torus orbit $T$. Klyachko's construction guarantees that these extend regularly to $\operatorname{GL}(W)$-valued transition functions on overlaps $X_\tau \cap X_\sigma$. Because $G$ is a closed algebraic subgroup of $\operatorname{GL}(W)$ and $T$ is dense in $X_\tau \cap X_\sigma$, these regular extensions take values in $G$ throughout. Gluing the local trivial bundles $X_\sigma \times G$ via $g_{\tau\sigma}$ constructs the required framed toric principal $G$-bundle $\mathcal{P}$ with ray parabolics $\{P_\rho\}_{\rho \in \Sigma(1)}$.

For the splitting assertion, the apartment criterion \cite[Corollary~5.5]{HK} states that $\mathcal{P}$ splits equivariantly if and only if its ray cocharacter classes admit representatives in a single maximal torus $A \subset G$. A class represented by $\lambda_\rho$ has a representative in $A$ if and only if $A \subseteq P(\lambda_\rho) = P_\rho$, as the maximal tori of $P_\rho$ are $P_\rho$-conjugate. Thus, all ray cocharacter classes admit representatives in $A$ if and only if 
\[
  A \subseteq \bigcap_{\rho \in \Sigma(1)} P_\rho,
\]
which holds precisely when the ray parabolics share a common maximal torus.
\end{proof}

\begin{lem}
\label{lem:all-obstruction-sizes}
For every integer $q$ with $3\leq q\leq h(G)$ there exist parabolic
subgroups $P_1,\ldots,P_q$ such that every proper subfamily has a
common maximal torus, but the whole family does not.
\end{lem}

\begin{proof}
Set $h = h(G)$. By the minimality of $h$ and the definition of the Helly property, there exists an inclusion-minimal family of exactly $h$ parabolic subgroups of $G$ without a common maximal torus. By Lemma~\ref{lem:toric-construction}, this family is realized as the collection of ray parabolic subgroups of a non-split toric principal $G$-bundle $\mathcal{E}$ on $\mathbb{P}^{h-1}$.

Choose a standard coordinate subspace $Y \cong \mathbb{P}^{q-1} \subset \mathbb{P}^{h-1}$. Equipping the restricted bundle $\mathcal{E}|_Y$ with the action of the torus complementary to the kernel of the ambient torus action on $Y$ endows $\mathcal{E}|_Y$ with the structure of a toric principal $G$-bundle over $Y$. This restricted bundle $\mathcal{E}|_Y$ must be non-split; otherwise, as $\dim (Y) = q-1 \ge 2$, Lemma~\ref{lem:principal_horrocks} would force $\mathcal{E}$ itself to split equivariantly over $\mathbb{P}^{h-1}$, yielding a contradiction.

The fan of $Y \cong \mathbb{P}^{q-1}$ contains $q$ rays, any proper subset of which spans a cone in the fan. Consequently, the $q$ ray parabolic subgroups associated with $\mathcal{E}|_Y$ satisfy the required property.
\end{proof}

\begin{theorem}[Generalization of Klyachko's fan criterion]
\label{thm:klyachko-fan-principal}
Let $X_\Sigma$ be a smooth toric variety, and put
$s=|\Sigma(1)|$. If $s\leq2$, every toric principal $G$-bundle
on $X_\Sigma$ splits.

Suppose $s\geq3$, and set $q=\min\{s,h(G)\}$.
The following conditions are equivalent:
\begin{enumerate}
  \item every toric principal $G$-bundle on $X_\Sigma$ splits;
  \item every set of $q$ distinct rays of $\Sigma$ is contained
        in a cone of $\Sigma$.
\end{enumerate}
Since $\Sigma$ is smooth, the second condition equivalently says
that those rays generate a cone of $\Sigma$.
In particular, if $s\geq h(G)$, one may take $q=h(G)$.
\end{theorem}

\begin{proof}
If $G$ is a torus, every principal $G$-bundle is already
reduced to a maximal torus, and the assertion is immediate. We therefore assume that $G$ has positive semisimple rank.

For $s\leq2$, all ray parabolics have a common maximal torus. Now suppose $s\geq3$ and the second condition holds. The ray parabolics of any toric principal bundle have a common maximal torus on every $q$-element subfamily. If $s<h(G)$, then $q=s$, so this already concerns the whole family. If $s\geq h(G)$, the defining property of $h(G)$ gives the same conclusion: subfamilies of smaller size can be enlarged to size $h(G)$. The bundle therefore splits equivariantly.

Conversely, suppose that $q$ distinct rays
$\rho_1,\ldots,\rho_q$ are not contained in a cone.
Choose a minimally incompatible family $P_1,\ldots,P_q$
as in Lemma~\ref{lem:all-obstruction-sizes}. Assign $P_i$
to $\rho_i$, and assign $G$ to every other ray.
Each cone contains only a proper subset of the selected rays,
so condition~\eqref{eq:ray-parabolic-compatibility} holds.
The resulting toric principal bundle is nonsplit because all
its ray parabolics have no common maximal torus.
\end{proof}

\begin{cor}[Generalization of Klyachko's dimension criteria]
\label{cor:klyachko-dimension-principal}
Let $X_\Sigma$ be a smooth $n$-dimensional toric variety with
$|\Sigma(1)|\geq3$, and put $h=h(G)$.
\begin{enumerate}
  \item If $n=h-1$, every toric principal $G$-bundle on
        $X_\Sigma$ splits if and only if $X_\Sigma$ is affine.
  \item If $n=h$, every toric principal $G$-bundle on
        $X_\Sigma$ splits if and only if $X_\Sigma$ is affine
        or is torically isomorphic to $\mathbb P^n$.
\end{enumerate}
\end{cor}

\begin{proof}
Set $s = |\Sigma(1)|$. Since $\Sigma$ is a simplicial $n$-dimensional fan, every cone $\sigma \in \Sigma$ contains at most $n$ rays. 

When $n = h - 1$ and $s \ge h$, the condition of Theorem~\ref{thm:klyachko-fan-principal} cannot be satisfied because no single cone contains $h$ rays. Conversely, when $s < h$, the condition forces all $s$ rays to lie in a single cone, which is equivalent to $X_\Sigma$ being affine. This completes the proof of the first assertion.

Now suppose $n = h$. If $s \le n$, the criterion requires all $s$ rays to lie in a single cone, which implies that $X_\Sigma$ is affine.

Suppose instead $s > n$. The condition requires every subset of $n$ rays to span a maximal cone in $\Sigma$. Select one such cone $\sigma_0$ and write its primitive generators as a lattice basis $e_1, \dots, e_n$ of $N$, using the smoothness of $X_\Sigma$.

For any additional primitive ray generator $v \in \Sigma(1) \setminus \{e_1, \dots, e_n\}$, each set
\[
    \{v, e_1, \dots, \widehat{e}_i, \dots, e_n\} \qquad (1 \le i \le n)
\]
must also generate a smooth maximal cone $\sigma_i \in \Sigma$. The fan intersection axiom dictates that $\sigma_i$ and $\sigma_0$ intersect at their common facet $\operatorname{cone}(e_1, \dots, \widehat{e}_i, \dots, e_n)$. Expressing $v = \sum\limits_{j=1}^n c_j e_j$, this separation forces the $i$-th coordinate $c_i$ to be strictly negative. Because $\sigma_i$ is smooth, its generators form a lattice basis, implying $|c_i| = 1$. Consequently, $c_i = -1$ for all $1 \le i \le n$, which uniquely determines
\[
    v = -e_1 - \dots - e_n.
\]
Thus, there is at most one additional ray ($s = n + 1$), and the required cones are precisely the maximal cones of the fan of $\mathbb{P}^n$.

Conversely, if $X_\Sigma$ is affine, all rays lie in a single cone. If $X_\Sigma \cong \mathbb{P}^n$, every collection of $n$ rays out of the $n+1$ total rays spans a maximal cone of the fan.
\end{proof}

\begin{rmk}
\label{rmk:klyachko-small-fans}


For $G = \operatorname{GL}_m$ with $m \ge 2$, the Helly number is $h(G) = m + 1$. Thus, Theorem~\ref{thm:klyachko-fan-principal} recovers \cite[Corollary~6.1.4]{Klyachko} when $s \ge m + 1$, while Corollary~\ref{cor:klyachko-dimension-principal} yields the two  splitting criteria of \cite[Corollary~6.1.5]{Klyachko}. 
\end{rmk}

\begin{cor}
\label{cor:exact-projective-splitting}
For $d \ge 2$, every toric principal $G$-bundle on $\mathbb{P}^d$ splits equivariantly if and only if $d \ge h(G)$. Moreover, every toric principal $G$-bundle on $\mathbb{P}^1$ splits equivariantly.
\end{cor}

\begin{proof}
The fan of $\mathbb{P}^d$ has $d + 1$ rays, and every proper subset of rays spans a cone, whereas the full set of $d + 1$ rays does not. The result then follows directly from Theorem~\ref{thm:klyachko-fan-principal}.
\end{proof}

\begin{rmk}
    The sufficient direction was established in \cite[Corollary 6.5]{HK}; the converse follows from the argument above.
\end{rmk}

The following result is an immediate consequence of the bundle-theoretic construction established in the proof of Theorem~\ref{thm:parabolicsGen}.

\begin{cor}
\label{cor:semisimple-rank-splitting}
Let $G$ be a connected reductive complex algebraic group of semisimple rank $m = \operatorname{rank}(G^{\mathrm{der}})$. If $d \ge m + 2$, then every toric principal $G$-bundle on $\mathbb{P}^d$ splits equivariantly. 
\end{cor}

\begin{prop}
\label{prop:levi-rank-helly}
Let $G$ be a connected reductive complex algebraic group, and let $\mathcal{P}$ be a collection of parabolic subgroups of $G$ whose Levi factors have semisimple rank at most $\ell$. A finite family of members of $\mathcal{P}$ shares a common maximal torus if every subcollection of at most $\ell+3$ members shares a common maximal torus.
\end{prop}

\begin{proof}
Suppose, by way of contradiction, that the assertion fails, and let $\{P_0, P_1, \dots, P_n\} \subset \mathcal{P}$ be an inclusion-minimal family that does not contain a common maximal torus.

The hypothesis implies that $n + 1 \ge \ell + 4$ (so $n \ge \ell + 3$), while minimality guarantees that every proper subcollection shares a common maximal torus. Applying Lemma~\ref{lem:toric-construction}, this subfamily compatibility allows us to construct a toric principal $G$-bundle $\mathcal{P}$ over $\mathbb{P}^n$ whose ray parabolics are precisely $P_0, \dots, P_n$.

Consider the coordinate hyperplane $D \cong \mathbb{P}^{n-1} \subset \mathbb{P}^n$ corresponding to $P_0$. The centralizer of the ray-generating cocharacter associated with $P_0$ is a Levi factor $L_0 \subset P_0$, which by hypothesis has semisimple rank at most $\ell$. By Lemma~\ref{lem:centralizer_restriction}, the restriction $\mathcal{P}|_D$ admits a reduction of structure group to $L_0$.

Since $\dim (D) = n - 1 \ge \ell + 2$, Corollary~\ref{cor:semisimple-rank-splitting} implies that the reduced principal $L_0$-bundle over $D$ splits equivariantly for the complementary torus. In particular, the restricted bundle $\mathcal{P}|_D$ splits as an ordinary principal $G$-bundle. 

As $\dim (D) \ge 2$, Lemma~\ref{lem:principal_horrocks} guarantees that $\mathcal{P}$ itself splits equivariantly over all of $\mathbb{P}^n$. By Lemma~\ref{lem:toric-construction}, this implies that the entire family $\{P_i\}$ shares a common maximal torus, contradicting our initial assumption.

\end{proof}

Since Borel subgroups are parabolic subgroups whose Levi factors are maximal tori (and thus have semisimple rank $\ell = 0$), setting $\ell = 0$ in Proposition~\ref{prop:levi-rank-helly} immediately yields the following generalization of Corollary~6.2.2 in \cite{Klyachko}.

\begin{cor}
\label{cor:borels}
Let $G$ be a connected reductive complex algebraic group, and let $\{B_\alpha\}_{\alpha\in A}$ be a finite family of Borel subgroups of $G$. If every subfamily of at most three members contains a common maximal torus of $G$, then the entire intersection $\bigcap\limits_{\alpha\in A} B_\alpha$ contains a common maximal torus of $G$.
\end{cor}






As another application of this framework, we obtain an exact Helly-type threshold for families of partial flags in $\mathbb{C}^m$ parametrized by the number of their nonzero successive quotients.

\begin{cor}
\label{cor:partial-flag-length-helly}
Fix integers $2 \le s \le m$. For any finite collection of partial flags in $\mathbb{C}^m$ each having at least $s$ nonzero successive quotients (i.e., at least $s$ steps), the optimal threshold for simultaneous adaptation to a common basis is
\[
  m - s + 3.
\]
In particular:
\begin{itemize}
  \item setting $s = 2$ (flags with two quotients, corresponding to single subspaces) recovers Klyachko's threshold of $m + 1$;
  \item setting $s = m$ (complete flags, corresponding to Borel subgroups) recovers the threshold of $3$.
\end{itemize}
\end{cor}

\begin{proof}
If a flag has $t\geq s$ successive quotients of dimensions
$d_1,\ldots,d_t$, its stabilizer has Levi factor
$\prod\limits_{j=1}^t\operatorname{GL}_{d_j}$, whose semisimple rank is given by 
\[
  \operatorname{rank}\left(
    \left(\prod_{j=1}^t\operatorname{GL}_{d_j}\right)^{\mathrm{der}}
  \right)
  =\sum_{j=1}^t(d_j-1)=\left(\sum_{j=1}^td_j\right)-t=m-t\le m-s.
\]
 Hence, Proposition~\ref{prop:levi-rank-helly} implies the upper bound.

For the lower bound put $d=m-s+2$, choose a $d$-dimensional
subspace $W\subseteq\mathbb C^m$, and take $d+1$ lines in $W$
such that every $d$ are independent. Complete each line to a flag
with dimension sequence
\[
  0,1,d,d+1,\ldots,m,
\]
using the same $W$ and the same subspaces above it. These flags
have exactly $s$ nonzero successive quotients. Any $d$ admit a
common adapted basis, but all $d+1$ do not: a common basis would
make $W$ coordinate and would give $d+1$ distinct coordinate
lines inside a $d$-dimensional space. Thus the threshold is at
least $d+1=m-s+3$.
\end{proof}


\section{Interpretation in Terms of Buildings}
\label{sec:buildings}

We recall the building-theoretic language needed to interpret the
preceding results; see \cite{AB} for general background. A simplicial
building is a simplicial complex equipped with subcomplexes called
\emph{apartments}, each isomorphic to a fixed Coxeter complex, such
that any two simplices lie in an apartment and any two apartments
admit an isomorphism fixing their intersection pointwise.
Its maximal simplices are called \emph{chambers}.
Spherical apartments are finite Coxeter complexes.

\begin{definition}
\label{defn:apartment-helly}
The \emph{apartment Helly number} $h(\Delta)$ of a building $\Delta$ with a specified apartment system is the smallest positive integer $h$ such that every finite family $\mathcal{S}$ of simplices is contained in a common apartment whenever every subfamily of size at most $h$ is contained in a common apartment. If no such integer exists, we set $h(\Delta) = \infty$. 

\end{definition}

Let $\mathcal{A}$ denote the collection of apartments in $\Delta$. For any simplex $\sigma \in \Delta$, let
\[
  \mathcal{A}(\sigma) \coloneqq \{ A \in \mathcal{A} \mid \sigma \subseteq A \}
\]
be the subcollection of apartments containing $\sigma$. A family $\mathcal{S} \subseteq \Delta$ of simplices lies in a common apartment if and only if
\[
  \bigcap_{\sigma \in \mathcal{S}} \mathcal{A}(\sigma) \neq \varnothing.
\]

Thus $h(\Delta)$ is the classical Helly number of this incidence family in the set of apartments.

\subsection{Spherical buildings and common maximal tori}

Let $G$ be a connected reductive complex algebraic group, and put
$r=\operatorname{rank}(G^{\mathrm{der}})>0$.
The spherical building $\Delta_{\mathrm{sph}}(G)$ has proper
parabolic subgroups as simplices, ordered by reverse inclusion.
Maximal proper parabolics correspond to vertices, while Borel
subgroups correspond to chambers. A maximal torus $T\subseteq G$
determines the apartment
\[
  A_T=\{P\subsetneq G\mid\ P\text{ is parabolic and }T\subseteq P\}.
\]
Consequently, parabolics $P_1,\ldots,P_N$ correspond to simplices in one apartment if and only if $\bigcap\limits_iP_i$ contains a maximal torus of $G$. An occurrence of $P_i=G$ imposes no condition and may
be omitted.

Translating the preceding results into building-theoretic terms yields global Helly bounds for the spherical building associated with $G$.

\begin{cor}
\label{cor:spherical-building-helly}
Let $G$ be a connected almost simple complex algebraic group of positive semisimple rank. Then
\[
    h\bigl(\Delta_{\mathrm{sph}}(G)\bigr) = h(G).
\]
In particular, the exact values and bounds for $h\bigl(\Delta_{\mathrm{sph}}(G)\bigr)$ coincide with those listed in Corollary~\ref{cor:helly-values-by-type}, leaving open cases only for types $D_m$ and $E_6$.
\end{cor}

\begin{proof}
The simplices of $\Delta_{\mathrm{sph}}(G)$ correspond to proper parabolic subgroups of $G$, and a collection of simplices lies in a common apartment if and only if the corresponding parabolic subgroups share a common maximal torus. Hence,
\[
    h\bigl(\Delta_{\mathrm{sph}}(G)\bigr) = h(G).
\]


\end{proof}

\end{document}